\documentclass{amsart}
        \usepackage{latexsym}
        \usepackage{amssymb}
        \usepackage{amsmath}
        \usepackage{amsfonts}
        \usepackage{amsthm}
        \usepackage[hypertexnames=false]{hyperref}
        \usepackage[all,knot,poly,2cell]{xy}
             \UseAllTwocells
        \usepackage{xspace}
        \usepackage{eucal}

        \theoremstyle{plain}
        \newtheorem{thm}{Theorem}[section]
        \newtheorem{cor}[thm]{Corollary}
        \newtheorem{lem}[thm]{Lemma}
        \newtheorem{prop}[thm]{Proposition}
        \newtheorem{mainthm}{Theorem}

        \theoremstyle{definition}
        
        \newtheorem{ex}[thm]{Example}
        \newtheorem{notn}[thm]{Notation}
 
        \theoremstyle{remark}
        \newtheorem{rem}[thm]{Remark}

        \numberwithin{equation}{section}
 
        \newcommand{\suchthat}{\,:\,}
        \newcommand{\itemref}[1]{\eqref{#1}}
        \newcommand{\opcit}[1][]{[\emph{op.\ cit.}{#1}]\xspace}

        \newcommand{\Z}{\mathbb{Z}}

        \newcommand{\Orb}{\mathcal{O}}    
       \DeclareMathOperator{\spec}{Spec}
        
        \newcommand{\COH}[1]{\mathbf{Coh}({#1})}
        \newcommand{\QCOHB}{\mathbf{QCoh}}
        \newcommand{\QCOH}[2][]{\QCOHB^{#1}({#2})}
        
        \newcommand{\et}{\mathrm{\acute{e}t}}
        \newcommand{\Et}{\mathrm{\acute{E}t}}
        
        \newcommand{\Homstk}{\underline{\Hom}}

        \newcommand{\MOD}[1]{\mathbf{Mod}({#1})}          
        \newcommand{\ideal}{\triangleleft}      

        \DeclareMathOperator{\coker}{coker}
        
        \DeclareMathOperator{\Hom}{Hom}
        \DeclareMathOperator{\Ext}{Ext}
        \newcommand{\SRExt}[4]{\SExt^{{#1}}({#2};{#3},{#4})}
        \DeclareMathOperator{\Tor}{Tor}
        \DeclareMathOperator{\supp}{supp}
        \newcommand{\COHO}[1]{\mathcal{H}^{{#1}}}
        \newcommand{\trunc}[1]{\tau^{{#1}}}
        \newcommand{\RDERF}{\mathsf{R}}
        \newcommand{\LDERF}{\mathsf{L}}
        
        \newcommand{\DCAT}{\mathsf{D}}
        \newcommand{\RHom}{\RDERF\!\Hom}
        \newcommand{\SHom}{\mathcal{H}om}
        \newcommand{\SRHom}{\RDERF\SHom}

        \newcommand{\ID}[1]{\mathrm{Id}_{#1}}    
        \newcommand{\tensor}{\otimes}
        \newcommand{\homotopic}{\simeq}

        \newcommand{\AB}{\mathbf{Ab}}
        \newcommand{\SETS}{\mathbf{Sets}}        
        \newcommand{\Sch}{\mathbf{Sch}}
        \newcommand{\SCH}[2][]{\Sch^{#1}/{#2}}

\newcommand{\Van}{\mathbb{V}}

\newcommand{\shv}[1]{\mathcal{{#1}}}
\newcommand{\fndefn}[1]{\emph{{#1}}}
\newcommand{\FUNP}[1]{\mathbf{Fun}_{{#1}}^{\times}}
\newcommand{\LIN}[1]{\mathbf{Lin}_{{#1}}}
\newcommand{\SExt}{\mathcal{E}xt}

\newcommand{\holim}[1]{\underset{#1}{\mathrm{holim}}\,}
\newcommand{\E}[2]{\mathbb{H}_{{#1},{#2}}}
\newcommand{\spref}[1]{\href{http://stacks.math.columbia.edu/tag/#1}{#1}}
\newcommand{\QCPSH}[1]{{#1}_{*}}
\newcommand{\QCPBK}[1]{{#1}_{\mathrm{Q}}^*}
\newcommand{\DQCOH}{\DCAT_{\mathrm{QCoh}}}
\newcommand{\DCOH}{\DCAT_{\mathrm{Coh}}}
\newcommand{\TSUB}[2]{\mathcal{T}_{{#1}}^{{#2}}}
\newcommand{\CSUB}[1]{\mathcal{C}_{{#1}}}

\newcommand{\sbt}{\,\begin{picture}(-1,-1)(0,-2)\circle*{1.5}\end{picture}\ }
\newcommand{\cplx}[1]{\shv{{#1}}^{\sbt}}
\newcommand{\DUAL}[1]{\mathfrak{D}({#1})}
\title{Cohomology and base change for algebraic stacks} 
\author{Jack Hall}
\date{March 8, 2013}
\begin{document}
\begin{abstract} 
  We prove that cohomology and base change holds for algebraic stacks, generalizing work of Brochard in the tame case. We also show that Hom-spaces on algebraic stacks are represented by abelian cones, generalizing results of Grothendieck, Brochard, Olsson, Lieblich, and Roth--Starr. To accomplish all of this, we prove that a wide class of Ext-functors in algebraic geometry are coherent (in the sense of M.~Auslander).
\end{abstract}
\address{Department of Mathematics, KTH, 100 44 Stockholm, Sweden}
\email{jackhall@kth.se}
\maketitle
\section*{Introduction}
Our first main result is the following version of
cohomology and base change.
\begin{mainthm}\label{mainthm:cohobc}
  Fix a morphism of locally noetherian algebraic stacks $f : X \to S$
  which is separated and locally of finite type, $\cplx{M} \in
  \DCOH^-(X)$, and $\shv{N} \in \COH{X}$ which is properly supported
  and flat over $S$. For each integer $q\geq 0$ and morphism of
  noetherian algebraic stacks $g : T \to S$, there is a
  natural base change morphism:
  \[
  b^q(T) : g^*\SRExt{q}{f}{\cplx{M}}{\shv{N}} \to \SRExt{q}{f_T}{\LDERF \QCPBK{(g_X)}\cplx{M}}{g_X^*\shv{N}},
  \]
  where $f_T : X_T \to T$  denotes the pullback of $f$ by $g$ and $g_X
  : X_T \to X$ denotes the pullback of $g$ by $f$. Now fix $s\in |S|$ such that $b^q(s)$ is surjective. 
  \begin{enumerate}
  \item Then, there exists an open neighbourhood $U \subseteq S$
    of $s\in |S|$ such that for any $g : T \to S$ factoring through
    $U$, the map $b^q(T)$ is an isomorphism.
  \item The following are equivalent:
    \begin{enumerate}
    \item $b^{q+1}(s)$ is surjective,
    \item the coherent $\Orb_S$-module
      $\SRExt{q+1}{f}{\cplx{M}}{\shv{N}}$ is free at $s$. 
    \end{enumerate}
  \end{enumerate}
\end{mainthm}
In the proof of Theorem \ref{mainthm:cohobc} for projective morphisms
\cite[III.7.7.5]{EGA}, the 
essential point was to corepresent the 
relevant functors by bounded above complexes of coherent sheaves on
$S$. Thus the following notion was indispensible: for a 
noetherian scheme $S$, a functor $\QCOH{S} \to \QCOH{S}$ is
\fndefn{corepresentable by a complex} if it is of the form $\shv{J}
\mapsto \COHO{0}(\SRHom_{\Orb_S}(\cplx{Q},\shv{J}))$ where $\cplx{Q} \in
\DCOH^-(S)$. 

Brochard noticed \cite[App.~A]{MR2871149}, however,
that corepresenting the cohomology functors of a non-tame stack by
complexes is quite subtle---the problem is that these stacks tend to
have infinite cohomological dimension. In particular, the ``Mumford Lemma''
\cite[Lem.~5.1]{MR0282985} is no longer applicable. Our next main
result shows that this problem can be circumvented when one has duality.
\begin{mainthm}\label{mainthms:coh_stk_dual}
  Fix a noetherian algebraic stack $S$ which admits a dualizing
  complex $\cplx{K}$, and a morphism of algebraic stacks $f : X
  \to S$ which is proper. Let $\cplx{M} \in  
  \DCOH^-(X)$ and $\cplx{N} \in
  \DCOH^b(X)$. If $\cplx{N}$ has finite tor-dimension over $S$, 
  then there exists a quasi-isomorphism:
  \begin{equation*}
    \RDERF \QCPSH{f}\SRHom_{\Orb_X}(\cplx{M},\cplx{N}
    \tensor_{\Orb_X}^{\LDERF } \LDERF \QCPBK{f}\cplx{I}) \homotopic
    \SRHom_{\Orb_S}(\cplx{E}_{\cplx{M},\cplx{N}},\cplx{I}) \quad
    \forall\, \cplx{I} \in
    \DQCOH^+(S),
  \end{equation*}  
  natural in $\cplx{I}$, where:
  \[
  \cplx{E}_{\cplx{M},\cplx{N}} \homotopic
  \SRHom_{\Orb_S}(\RDERF \QCPSH{f}\SRHom_{\Orb_X}(\cplx{M},\cplx{N}
  \tensor_{\Orb_X}^{\LDERF } \LDERF
  \QCPBK{f}\cplx{K}),\cplx{K}) \in \DCOH^-(S). 
  \]
  In addition, if $\cplx{N} \homotopic \shv{N}[0]$, where $\shv{N} \in
  \COH{X}$ is flat over $S$, then the formation of
  $\cplx{E}_{\cplx{M},\cplx{N}}$ is compatible with base change. 
\end{mainthm}
Since any separated scheme of finite type over a Gorenstein ring
(e.g.~$\spec \Z$ or a field) and the spectrum of any maximal-adically
complete noetherian local ring admits a dualizing complex. Thus, Theorem
\ref{mainthms:coh_stk_dual} covers most cases encountered in practice
(in particular, it generalizes \cite[Lem.~6.1]{MR1437495}). 

We wish to point out, however, that the collection of functors
$\QCOH{S} \to \QCOH{S}$ which are 
corepresentable by a complex is poorly behaved. Indeed, it is 
not even closed under direct summands \cite[Prop.~4.6 \& Ex.~5.5]{MR1656482}.

The following generalization of functors which are corepresentable by
a complex was considered by M.~Auslander \cite{MR0212070} in order to 
correct such deficiencies. For an
affine (not necessarily noetherian) scheme $S$, a functor $F :
\QCOH{S} \to \AB$ is \fndefn{coherent} if there exists a morphism of
quasicoherent $\Orb_S$-modules $\shv{K}_1 \to \shv{K}_2$, such that
for all $\shv{I}\in \QCOH{S}$, there is a natural isomorphism of
abelian groups:
\[
F(\shv{I}) \cong \coker(\Hom_{\Orb_S}(\shv{K}_2,\shv{I}) \to
\Hom_{\Orb_S}(\shv{K}_1,\shv{I})). 
\]
Note that if $S$ is noetherian, the coherent functors of R.~Hartshorne
\cite{MR1656482} are precisely those coherent functors of M.~Auslander
\cite{MR0212070} which preserve direct limits. 

In any case, the
collection of coherent functors is very well-behaved: it is an 
abelian category which is closed under extensions and inverse
limits---precisely the sort of properties that are convenient
to have at one's disposal when performing induction arguments. Thus, using
Theorem \ref{mainthms:coh_stk_dual}, we can prove the following Theorem. 
\begin{mainthm}\label{mainthms:cohstk_flat}
  Fix an affine scheme $S$, a morphism of algebraic stacks $f : X \to
  S$ which is separated and locally of finite presentation, 
  $\cplx{M}\in \DQCOH(X)$, and $\shv{N} \in \QCOH{X}$, with $\shv{N}$ of finite
  presentation, properly supported and flat over $S$. Then, the functor:
  \[
  \Hom_{\Orb_X}(\cplx{M}, \shv{N}\tensor_{\Orb_X}^\LDERF \LDERF \QCPBK{f}(-)) : \QCOH{S}
  \to \AB
  \]
  is coherent. 
\end{mainthm}
We wish to emphasize that Theorem \ref{mainthms:cohstk_flat}
eliminates from Theorem \ref{mainthms:coh_stk_dual} the finiteness
hypotheses on $S$ (i.e.~$S$ is permitted to be non-noetherian) and on
$\cplx{M}$ (i.e.~$\cplx{M}$ can be unbounded with quasicoherent
cohomology). Moreover, the employment of Auslander's definition of
coherent functors is essential for the proof and truth of Theorem
\ref{mainthms:cohstk_flat}. If $S$ is noetherian and $\cplx{M} \in
\DCOH^-(X)$, then Theorem \ref{mainthms:cohstk_flat}
implies that the functor is coherent in the sense of Hartshorne (see Lemma \ref{lem:lp_stk_hom}). 
 
Theorem \ref{mainthm:cohobc} is proved by combining a clever
result of A.~Ogus and G.~Bergman \cite[Cor.~5.1]{MR0302633},
Theorem \ref{mainthms:cohstk_flat}, and some general vanishing results
for coherent functors. An interesting application of Theorem
\ref{mainthms:cohstk_flat} is the following:
given a scheme $S$, an algebraic $S$-stack $X$, and $\shv{M}$, $\shv{N}
\in \QCOH{X}$, we define the $S$-presheaf
$\Homstk_{\Orb_{X}/S}(\shv{M},\shv{N})$ as follows:
\begin{align*}
  \Homstk_{\Orb_X/S}(\shv{M},\shv{N})[T\xrightarrow{\tau} S] &=
  \Hom_{\Orb_{X_T}}(\tau_X^*\shv{M},\tau_X^*\shv{N}),
\end{align*}
where $\tau_X : X \times_S T \to X$ is the projection. Then, we prove
\begin{mainthm}\label{mainthms:bdd_hom_shf}
  Fix a scheme $S$ and a morphism of algebraic stacks $f :X \to S$, which
  is separated and locally of finite presentation. Let $\shv{M}$,
  $\shv{N} \in \QCOH{X}$, with $\shv{N}$ of finite presentation, flat
  over $S$, with support proper over $S$. Then,
  $\Homstk_{\Orb_X/S}(\shv{M},\shv{N})$ is  
  representable by an abelian cone over $S$ (which is, in particular,
  affine over $S$). If $\shv{M}$ is of finite presentation, then 
  $\Homstk_{\Orb_X/S}(\shv{M},\shv{N})$ is finitely presented over $S$. 
\end{mainthm}
We wish to emphasize that Theorem \ref{mainthms:bdd_hom_shf} is
completely elementary once Theorem \ref{mainthms:cohstk_flat} is
known. Using different techniques to that employed in Theorem
\ref{mainthms:cohstk_flat}, we prove a coherence result when nothing
is assumed to be flat (at the expense of making the diagonal finite). 
\begin{mainthm}\label{mainthms:cohstk_noeth_fd}
  Fix an affine and noetherian scheme $S$, a morphism of algebraic
  stacks $f : X \to S$ which is locally of finite type with finite
  diagonal, $\cplx{M} \in \DQCOH(X)$, and $\cplx{N}
  \in \DCOH^b(X)$. If $\cplx{N}$ has properly supported cohomology sheaves 
  over $S$, then the functor: 
  \[
  \Hom_{\Orb_X}(\cplx{M},\cplx{N}
  \tensor_{\Orb_X}^\LDERF \LDERF \QCPBK{f}(-)) : \QCOH{S} \to \AB
  \]
  is coherent. 
\end{mainthm}
We wish to emphasize that Theorem \ref{mainthms:cohstk_noeth_fd} is
 independent of Theorem \ref{mainthms:coh_stk_dual}.
\subsection*{Relation with other work}
In \cite{hallj_openness_coh}, coherent functors featured prominently
in a criteria for algebraicity of stacks. Thus Theorem
\ref{mainthms:cohstk_flat} can be used to show that certain stacks are
algebraic \opcit[, \S\S8--9]. Theorem \ref{mainthms:bdd_hom_shf} can
also be used to show that many algebraic stacks of interest have affine
diagonals \cite[\S\S8-9]{hallj_openness_coh}---generalizing and
simplifying the existing work of M.~Olsson
\cite[Prop.~5.10]{MR2239345}, M.~Lieblich \cite[Prop.~2.3]{MR2233719},
and M.~Roth and J.~Starr \cite[Thm.~2.1]{2009arXiv0908.0096R}. In
\cite{hallj_dary_g_hilb_quot}, Theorem \ref{mainthms:bdd_hom_shf} is
combined with the absolute approximation results of D. Rydh
\cite{rydh-2009} to show that Hilbert stacks and Quot spaces exist
without finiteness assumptions.

Results of A. Grothendieck \cite[III.7.7.8-9]{EGA}, A. Altman and
S. Kleiman \cite[1.1]{MR555258}, and S. Brochard \cite[Prop.\
A.4.3]{MR2871149} demonstrate that Theorem \ref{mainthm:cohobc} holds
when $f$ is proper, tame, and flat, and $\shv{M}$
is flat over $S$ as well as being the cokernel of a map vector bundles. For
example, if $f$ is projective, or more generally, $X$ is tame and has
the resolution property, Theorem \ref{mainthm:cohobc} is known to hold
for any coherent sheaf $\shv{M}$ which is flat over $S$.

After completing this paper we also located in the literature two very nice
papers of H.~Flenner addressing similar results for analytic
spaces. In particular, if $S$ is excellent and of finite Krull
dimension, then Theorem \ref{mainthm:cohobc} follows from the results
of \cite[\S7]{MR638811} and Theorems 
\ref{mainthms:coh_stk_dual} and \ref{mainthms:bdd_hom_shf} are the main results of \cite{MR641823}
(this is in the analytic category, however, thus $\shv{M}$ is
assumed to be coherent in Theorem \ref{mainthms:bdd_hom_shf}). Without
further assumptions on $f$ and $\cplx{M}$ (e.g. $\cplx{M}$ and $X$ are
flat over $S$ when $q\geq 1$ \cite[\S1]{MR555258}) we cannot see how
Theorem \ref{mainthm:cohobc} can be easily reduced to the case where
$S$ meets Flenner's hypotheses. In fact, we use coherent functors to
accomplish this descent, which is effectively the content of Theorem
\ref{mainthms:cohstk_flat}. This has no counterpart in the analytic
category where everything is excellent, coherent, and admits a
dualizing complex. We do not believe that Theorem
\ref{mainthms:cohstk_noeth_fd} has been considered previously.
\subsection*{Assumptions, conventions, and notations}
For a scheme $T$, denote by $|T|$ the underlying topological space
(with the Zariski topology) and $\Orb_T$ the (Zariski) sheaf of rings
on $|T|$. For $t\in |T|$, let $\kappa(t)$ denote the residue field.
Denote by $\QCOH{T}$ (resp.\ $\COH{T}$) the abelian category of
quasicoherent (resp.\ coherent) sheaves on the scheme $T$. Let
$\SCH{T}$ denote the category of schemes over $T$. The big \'etale
site over $T$ will be denoted by $(\SCH{T})_\Et$.  

For a ring $A$ and an $A$-module $M$, denote the quasicoherent
$\Orb_{\spec A}$-module associated to $M$ by $\widetilde{M}$. Denote
the abelian category of all (resp.\ coherent) $A$-modules by $\MOD{A}$
(resp.\ $\COH{A}$). 

We will assume throughout that all schemes, algebraic spaces, and
algebraic stacks have quasicompact and quasiseparated diagonals. 
\subsection*{Acknowledgements}  
I would like to thank R. Ile and D. Rydh for several interesting and
supportive discussions.  
\tableofcontents
\section{Derived categories of sheaves on algebraic stacks}
In this section, we review derived categories of sheaves on algebraic
stacks. Fix an algebraic stack $X$. We take $\MOD{X}$ (resp.~$\QCOH{X}$) to
denote the abelian category of $\Orb_X$-modules (resp.~quasicoherent
$\Orb_X$-modules) on the lisse-\'etale site of $X$
\cite[12.1]{MR1771927}. Take $\DCAT(X)$ (resp.~$\DQCOH(X)$) to denote
the unbounded derived category of $\MOD{X}$ (resp.~the full
subcategory of $\DCAT(X)$ with cohomology in $\QCOH{X}$). Superscripts
such as $+$, $-$, $\geq n$, and $b$ decorating $\DCAT(X)$ and
$\DQCOH(X)$ should be interpreted as usual. In addition, if $X$ is
locally noetherian, one may consider the category of coherent sheaves
$\COH{X}$ and the derived category $\DCOH(X)$. 

If $X$ is a Deligne--Mumford stack, there is an
associated small \'etale site. We take $\MOD{X_{\et}}$
(resp.~$\QCOH{X_{\et}}$) to denote the abelian category of
$\Orb_{X_{\et}}$-modules (resp.~quasicoherent
$\Orb_{X_{\et}}$-modules). There are naturally induced morphisms of
abelian categories $\MOD{X} \to \MOD{X_{\et}}$ and $\QCOH{X} \to
\QCOH{X_{\et}}$. Set $\DQCOH(X_{\et})$ to be the triangulated category
$\DCAT_{\QCOH{X_{\et}}}(\MOD{X_{\et}})$. Then, the natural functor
$\DQCOH(X) \to \DQCOH(X_{\et})$ is an equivalence of
categories. If $X$ is a scheme, the corresponding statement for the
Zariski site is also true.

For generalities on
unbounded derived categories on ringed sites, we refer the reader to
\cite[\S18.6]{MR2182076}. We now record for future reference some useful
formulae. If $\cplx{M}$ and $\cplx{N} \in \DCAT(X)$,
then there is the derived tensor product $\cplx{M}
\tensor^{\LDERF}_{\Orb_X} \cplx{N} \in \DCAT(X)$, the derived sheaf
Hom functor $\SRHom_{\Orb_X}(\cplx{M},\cplx{N}) \in \DCAT(X)$ and the
derived global Hom functor $\RHom_{\Orb_X}(\cplx{M},\cplx{N}) \in
\DCAT(\AB)$. For all $\cplx{P}\in\DCAT(X)$ we have a functorial isomorphism:
\begin{equation}
  \Hom_{\Orb_X}(\cplx{M} \tensor^{\LDERF}_{\Orb_X} \cplx{N}, \cplx{P})
  \cong
  \Hom_{\Orb}(\cplx{M},\SRHom_{\Orb_X}(\cplx{N},\cplx{P})),\label{eq:ghomadj}  
\end{equation}
as well as a functorial quasi-isomorphism: 
\begin{equation}
  \SRHom_{\Orb_X}(\cplx{M} \tensor^{\LDERF}_{\Orb_X} \cplx{N},
  \cplx{P}) \homotopic
  \SRHom_{\Orb_X}(\cplx{M},\SRHom_{\Orb_X}(\cplx{N},\cplx{P}))\label{eq:lhomadj}. 
\end{equation}
Set $\RDERF \Gamma(X,-)  = \RHom_{\Orb_X}(\Orb_X,-)$, then there
is also a natural quasi-isomorphism:
\begin{equation}
  \RHom_{\Orb_X}(\cplx{M},\cplx{N}) \homotopic \RDERF \Gamma
  \SRHom_{\Orb_X}(\cplx{M},\cplx{N}). \label{eq:lghom} 
\end{equation}
If $\cplx{M}$, $\cplx{N}\in \DQCOH(X)$,
then $\cplx{M}\tensor^{\LDERF}_{\Orb_X}\cplx{N} \in \DQCOH(X)$. If
$X$ is locally noetherian and $\cplx{M}$, $\cplx{N} \in \DCOH^-(X)$,
then $\cplx{M}\tensor^{\LDERF}_{\Orb_X}\cplx{N} \in
\DCOH^-(X)$. Also, if $X$ is locally noetherian and $\cplx{M} \in
\DCAT^-(X)$ and $\cplx{N} \in \DQCOH^+(X)$ (resp.~$\DCOH^+(X)$), then
$\SRHom_{\Orb_X}(\cplx{M},\cplx{N}) \in \DQCOH^+(X)$
(resp.~$\DCOH^+(X)$). These results are all
consequences of \cite[\S6]{MR2312554} and \cite[\S2]{MR2434692}.

Fix a morphism of algebraic stacks $f : X \to Y$. We let $\RDERF f_* :
\DCAT(X) \to \DCAT(Y)$ denote the derived functor of $f_* : \MOD{X}
\to \MOD{Y}$. For $\cplx{M}$, $\cplx{N} \in \DCAT(X)$ and $q\in \Z$, set: 
\[
\SRExt{q}{f}{\cplx{M}}{\cplx{N}} := \COHO{q}(\RDERF
f_*\SRHom_{\Orb_X}(\cplx{M},\cplx{N})). 
\]
If the morphism $f$ is quasicompact and quasiseparated,
then the restriction of $\RDERF f_*$ to $\DQCOH^+(X)$ induces a
functor $\RDERF \QCPSH{f} : \DQCOH^+(X) \to \DQCOH^+(Y)$
\cite[Lem.~6.20]{MR2312554}.  If $X$ and $Y$ are Deligne--Mumford stacks, then
the restriction of $\RDERF f_*$ to $\DQCOH^+(X)$ coincides with the
restriction of the derived functor of $(f_{\et})_{*} : \MOD{X_{\et}} \to
\MOD{Y_{\et}}$ to $\DQCOH(X_{\et})$. A consequence of
\cite[Ex.~2.1.11]{MR2434692} is that if in addition $f$ is representable, then the
restriction of $\RDERF f_*$ to $\DQCOH(X)$ induces a functor $\RDERF
\QCPSH{f} : \DQCOH(X) \to \DQCOH(Y)$. 

A morphism of algebraic stacks $f : X \to Y$, however, does not
necessarily induce a left exact morphism of corresponding
lisse-\'etale sites \cite[5.3.12]{MR1963494}, thus the construction of
the correct derived functors of $f^* : \QCOH{Y} \to \QCOH{X}$ is
somewhat subtle. There are currently two approaches to constructing
these functors. The first, due to M.~Olsson \cite{MR2312554} and
Y.~Laszlo and M.~Olsson \cite{MR2434692}, uses cohomological
descent. The other approach appears in the Stacks Project
\cite{stacks-project}. The latter approach is more widely applicable,
but uses a completely different formulation (big sites) and requires
significant amounts of technology that many people may not be familiar
with. 

In this article, where there is always some finiteness at hand,
the approach of Olsson and Laszlo--Olsson is sufficient,
thus is the method that we will employ. In any case, there exists a
functor $\LDERF \QCPBK{f} :\DQCOH(Y) \to \DQCOH(X)$ such that
$\COHO{0}(\LDERF \QCPBK{f} \shv{I}[0]) \cong f^*\shv{I}$, whenever
$\shv{I} \in \QCOH{Y}$. In addition, if $f$ is quasicompact and quasiseparated, $\cplx{I}
\in \DQCOH(Y)$, and $\cplx{N} \in \DQCOH^+(X)$, then there is a
natural isomorphism:
\begin{equation}
\Hom_{\Orb_Y}(\cplx{I},\RDERF \QCPSH{f}
\cplx{N}) \cong \Hom_{\Orb_X}(\LDERF \QCPBK{f}
\cplx{I},\cplx{N}).\label{eq:global_trivial_duality}  
\end{equation}
Thus, $\LDERF \QCPBK{f}$ is the left adjoint of $\RDERF f_*$.  
In the situation where $X$ and $Y$ are Deligne--Mumford stacks, there
also exists a derived functor $\LDERF f_{\et}^* : \DQCOH(Y_{\et}) \to
\DQCOH(X_{\et})$. The restriction of $\LDERF f_{\et}^*$ to
$\DQCOH(Y_{\et})$ coincides with $\LDERF
\QCPBK{f}$. 

If $f : X \to Y$  is a quasicompact morphism of locally
noetherian algebraic stacks, $\cplx{I} \in \DCOH^-(Y)$, and $\cplx{N}
\in \DQCOH^+(X)$, the isomorphism \eqref{eq:global_trivial_duality} is
readily strengthened to a natural quasi-isomorphism
in $\DQCOH^+(Y)$:
\begin{equation}
  \label{eq:sheaf_trivial_duality}
  \SRHom_{\Orb_Y}(\cplx{I},\RDERF \QCPSH{f}\cplx{N}) \homotopic \RDERF
  f_* \SRHom_{\Orb_X}(\LDERF \QCPBK{f} \cplx{I},\cplx{N}). 
\end{equation}
 
If $f : X \to Y$ is a representable and quasicompact morphism of
algebraic stacks, then the isomorphism
\eqref{eq:global_trivial_duality} can be extended to all $\cplx{N} \in
\DQCOH(X)$. In this situation---which is covered in 
\cite{perfect_complexes_stacks}, generalizing
\cite[Prop.~5.3]{MR1308405}---there is also the projection formula, 
which gives a functorial quasi-isomorphism for all $\cplx{N} \in \DQCOH(X)$
and $\cplx{I} \in \DQCOH(Y)$:
\begin{equation}
  \label{eq:projection_formula}
  (\RDERF \QCPSH{f}\cplx{N})\tensor^{\LDERF}_{\Orb_Y} \cplx{I}\homotopic
  \RDERF \QCPSH{f}(\cplx{N} \tensor^{\LDERF}_{\Orb_X} \LDERF\QCPBK{f}\cplx{I}).
\end{equation}
For a
morphism of algebraic stacks $f : X \to Y$ a complex $\cplx{N} \in
\DQCOH^-(X)$ has \fndefn{finite tor-dimension over $Y$} if there
exists a non-negative integer $n$ such that for all $i\in \Z$ and $\cplx{I} \in \DQCOH^{\geq
  i}(Y)$ we have that $\cplx{N}\tensor_{\Orb_X}^\LDERF \LDERF
\QCPBK{f}\cplx{I} \in \DQCOH^{\geq i-n}(X)$. We conclude this section with
the following easily proven lemma.
\begin{lem}\label{lem:lp_stk_hom}
  Fix an affine noetherian scheme $S$ and a morphism of noetherian
  algebraic stacks $f : X \to S$. Let
  $\cplx{M} \in  
  \DCOH^-(X)$ and $\cplx{N} \in
  \DCOH^b(X)$. If $\cplx{N}$ has finite tor-dimension over $S$, 
  then the following functor preserves filtered colimits:
  \[
  \Hom_{\Orb_X}(\cplx{M},\cplx{N}
  \tensor_{\Orb_X}^{\LDERF } \LDERF \QCPBK{f}(-)) : \QCOH{S} \to \AB.
  \]
\end{lem}
We now briefly review homotopy limits in a triangulated category
$\mathcal{T}$ admitting countable products. Fix for each $i\geq 0$ a
morphism in $\mathcal{T}$, $t_i : T_{i+1} \to 
T_i$. Set $t : \Pi_{i\geq 0} T_i \to \Pi_{i\geq 0} T_i$ to be
the composition of the product of the morphisms $t_i$ with the
projection $\Pi_{i\geq 0} T_i \to \Pi_{i\geq 1} T_i$. We define
$\holim{i} T_i$ via the following distinguished triangle:
\[
\xymatrix{\holim{i} T_i \ar[r] & \prod_{i\geq 0} T_i
  \ar[rr]^{\ID{}-t} & & \prod_{i\geq 0} T_i}.
\]
The category of lisse-\'etale $\Orb_X$-modules is a Grothendieck
abelian category, thus $\DCAT(X)$ admits small products. Moreover, we wish to 
point out that the functors $\RDERF \Gamma(X,-)$,
$\SRHom_{\Orb_X}(\cplx{M},-)$ for $\cplx{M} \in \DCAT(X)$, and $\RDERF
f_*$ for a morphism of algebraic stacks $f :X \to Y$ all preserve 
homotopy limits because they preserve products. 

The following result is well-known and appears in
\cite[\spref{08IY}]{stacks-project} (albeit in a slightly different,
but equivalent formulation). 
\begin{lem}\label{lem:lc_der_cat}
  Let $X$ be an algebraic stack and fix $\cplx{N}\in \DCAT(X)$. Then,
  the projections $\cplx{N} \to \trunc{\geq -i}\cplx{N}$ induce a
  non-canonical morphism:
  \[
  \phi : \cplx{N} \to \holim{i} \trunc{\geq -i}\cplx{N}.
  \]
  If $\cplx{N} \in \DQCOH(X)$, then any such $\phi$ is a quasi-isomorphism.
\end{lem}
Note that the main result of \cite{2011arXiv1103.5539N}
produces---in positive characteristic---complexes $\cplx{N} \in
\DCAT(\QCOH{B\mathbb{N}_a})$ with the property that there are no quasi-isomorphisms: 
\[
\cplx{N} \to \holim{i} \trunc{\geq -i} \cplx{N}.
\]
We wish to emphasize that this does not contradict Lemma 
\ref{lem:lc_der_cat}. Indeed, while the categories
$\DCAT^+(\QCOH{B\mathbb{G}_a})$ and $\DQCOH^+(B\mathbb{G}_a)$ are
equivalent \cite[Thm.~3.8]{lurie_tannaka}, this equivalence does not
extend to the unbounded derived categories. 
\section{Corepresentability of $\SRHom$-functors}
In this section we prove Theorem \ref{mainthms:coh_stk_dual}. Before
we get to this we require the
following two easily proven lemmas. Our first lemma gives two
important maps used to prove Theorem \ref{mainthms:coh_stk_dual}.
\begin{lem}\label{lem:big_one}
  Fix a morphism of noetherian algebraic stacks $f : X \to S$. Let
  $\cplx{N}\in \DCOH^-(X)$ have finite tor-dimension over $S$,
  $\cplx{F} \in \DCOH^-(S)$, and $\cplx{G} \in \DQCOH^+(S)$. 
  \begin{enumerate}
  \item \label{item:big_one:A}  There is a natural morphism in $\DQCOH^+(S)$:
    \[
    \SRHom_{\Orb_S}(\cplx{F},\cplx{G}) \to \RDERF
    f_*\SRHom_{\Orb_X}(\cplx{N}\tensor_{\Orb_X}^\LDERF \LDERF
    \QCPBK{f}\cplx{F}, \cplx{N}\tensor_{\Orb_X}^\LDERF \LDERF
    \QCPBK{f}\cplx{G}).
    \]
  \item \label{item:big_one:B} There is a natural quasi-isomorphism in
    $\DQCOH^+(X)$:
    \[
    \cplx{N}\tensor_{\Orb_X}^{\LDERF} \LDERF
    \QCPBK{f}\SRHom_{\Orb_S}(\cplx{F},\cplx{G}) \homotopic
    \SRHom_{\Orb_X}(\LDERF \QCPBK{f}
    \cplx{F},\cplx{N}\tensor_{\Orb_X}^{\LDERF} \LDERF \QCPBK{f}
    \cplx{G}).
    \]
  \end{enumerate}
\end{lem}
Our next lemma will give the compatibility of the corepresenting
object in Theorem \ref{mainthms:coh_stk_dual} with base change---the
result for schemes boils down to the well-known tor-independent base
change \cite[Thm.~8.3.2]{MR2222646}.
\begin{lem}\label{lem:base_change} 
Fix a $2$-cartesian square of noetherian algebraic stacks:
\[
\xymatrix@-0.8pc{X_T \ar[r]^{g_X} \ar[d]_{f_T} & X \ar[d]^{f} \\ T \ar[r]^{g}
& S.}
\]
Let $\cplx{M} \in \DCOH^-(X)$, $\shv{N} \in \COH{X}$, and $\cplx{I}
\in \DQCOH^+(T)$. If $\shv{N}$ is flat over $S$, then there is a natural
quasi-isomorphism in $\DQCOH^+(S)$:
\begin{align*}
  \RDERF \QCPSH{f} \SRHom_{\Orb_X}&(\cplx{M},\shv{N}[0]
  \tensor_{\Orb_X}^\LDERF \LDERF \QCPBK{f} \RDERF g_*\cplx{I})\\
  &\to \RDERF
  g_*\RDERF \QCPSH{(f_T)}\SRHom_{\Orb_{X_T}}(\LDERF
  \QCPBK{(g_X)}\cplx{M},
  (g_X^*\shv{N})[0]\tensor^{\LDERF}_{\Orb_{X_T}}
  \LDERF \QCPBK{(f_T)} \cplx{I} ).
\end{align*}
\end{lem}
We can now prove Theorem \ref{mainthms:coh_stk_dual}.
\begin{proof}[Proof of Theorem \ref{mainthms:coh_stk_dual}]
  For background material on dualizing complexes we refer the reader
  to \cite[V.2]{MR0222093}. For the
  convenience of the reader, however, we will recall the relevant
  results. A complex $\cplx{K} \in \DCOH^b(S)$ is dualizing if it is
  locally of finite injective dimension and for any $\cplx{F} \in
  \DCOH(S)$, the natural map:
  \[
  \cplx{F} \to
  \SRHom_{\Orb_S}(\SRHom_{\Orb_S}(\cplx{F},\cplx{K}),\cplx{K})
  \]
  is a quasi-isomorphism. For notational convenience we set 
  $\DUAL{-}=\SRHom_{\Orb_S}(-,\cplx{K})$. Two useful facts about
  dualizing complexes are the following:
  \begin{enumerate}
  \item \label{item:dualizing1} the functor $\DUAL{-}$ interchanges
    $\DCOH^-(S)$ and $\DCOH^+(S)$;
  \item \label{item:dualizing2} for $\cplx{F}$, $\cplx{G} \in
    \DCOH(S)$, there is a natural quasi-isomorphism:
    \[
    \SRHom_{\Orb_S}(\cplx{F},\cplx{G}) \homotopic
    \SRHom_{\Orb_S}(\DUAL{\cplx{G}},\DUAL{\cplx{F}}).
    \]
  \end{enumerate}
  Now fix $\cplx{I} \in \DCOH^+(S)$, then we have the following
  sequence of natural quasi-isomorphisms:
  \begin{align*}
    &\RDERF \QCPSH{f} \SRHom_{\Orb_X}(\cplx{M},\cplx{N}
    \tensor_{\Orb_X}^\LDERF \LDERF \QCPBK{f}\cplx{I}) \\
    &\homotopic \RDERF \QCPSH{f}\SRHom_{\Orb_X}(\cplx{M},\cplx{N}
    \tensor_{\Orb_X}^{\LDERF} \LDERF
    \QCPBK{f}\SRHom_{\Orb_S}(\DUAL{\cplx{I}},\cplx{K})) & \mbox{(by \itemref{item:dualizing1})}\\
    &\homotopic \RDERF
    \QCPSH{f}\SRHom_{\Orb_X}(\cplx{M},\SRHom_{\Orb_X} (\LDERF
    \QCPBK{f}\DUAL{\cplx{I}}, \cplx{N}
    \tensor_{\Orb_X}^{\LDERF} \LDERF \QCPBK{f}\cplx{K})) & \mbox{(by Lemma
      \ref{lem:big_one}\itemref{item:big_one:B})}\\
    &\homotopic \RDERF
    \QCPSH{f}\SRHom_{\Orb_X}(\cplx{M}\tensor^{\LDERF}_{\Orb_X} \LDERF
    \QCPBK{f}\DUAL{\cplx{I}},\cplx{N}
    \tensor_{\Orb_X}^{\LDERF} \LDERF \QCPBK{f}\cplx{K}) &\mbox{(by
      \eqref{eq:lhomadj})}\\ 
   &\homotopic \RDERF \QCPSH{f}\SRHom_{\Orb_X}(\LDERF
    \QCPBK{f}\DUAL{\cplx{I}},\SRHom_{\Orb_X} ( \cplx{M}, \cplx{N}
    \tensor_{\Orb_X}^{\LDERF} \LDERF \QCPBK{f}\cplx{K})) &\mbox{(by
      \eqref{eq:lhomadj})}\\ 
    &\homotopic \SRHom_{\Orb_S}(\DUAL{\cplx{I}},\RDERF
    \QCPSH{f}\SRHom_{\Orb_X} ( \cplx{M}, \cplx{N}
    \tensor_{\Orb_X}^{\LDERF} \LDERF \QCPBK{f}\cplx{K}))
    &\mbox{(by \eqref{eq:sheaf_trivial_duality})}\\ 
    &\homotopic \SRHom_{\Orb_S}(\DUAL{\RDERF \QCPSH{f}\SRHom_{\Orb_X}
      ( \cplx{M}, \cplx{N} \tensor_{\Orb_X}^{\LDERF} \LDERF
      \QCPBK{f}\cplx{K})}, \cplx{I}) & \mbox{(by \itemref{item:dualizing2})}.
  \end{align*}
  The final quasi-isomorphism is a consequence of the following
  sequence of observations. First, $\cplx{N} \in \DCOH^b(X)$ has
  finite tor-dimension over $S$ and $\cplx{K} \in \DCOH^b(S)$, so $\cplx{N}
  \tensor^{\LDERF}_{\Orb_X} \LDERF \QCPBK{f}\cplx{K} \in
  \DCOH^b(X)$. Also, $f$ is proper so \cite[6.4.4 \& 10.13]{MR2312554}
  implies that $\RDERF f_*\SRHom_{\Orb_X} ( \cplx{M}, \cplx{N}
  \tensor_{\Orb_X}^{\LDERF} \LDERF \QCPBK{f}\cplx{K}) \in
  \DCOH^+(S)$. Thus \itemref{item:dualizing2} applies.  Hence, we have
  produced a natural quasi-isomorphism for all $\cplx{I} \in \DCOH^+(S)$:
  \begin{equation}
    \RDERF \QCPSH{f}\SRHom_{\Orb_X}(\cplx{M},\cplx{N}
    \tensor_{\Orb_X}^{\LDERF } \LDERF \QCPBK{f}\cplx{I}) \homotopic
    \SRHom_{\Orb_S}(\cplx{E}_{\cplx{M},\cplx{N}},\cplx{I}). \label{eq:quasi_perfect}  
  \end{equation}  
  where
  \[
  \cplx{E}_{\cplx{M},\cplx{N}} \homotopic
  \SRHom_{\Orb_S}(\RDERF \QCPSH{f}\SRHom_{\Orb_X}(\cplx{M},\cplx{N}
  \tensor_{\Orb_X}^{\LDERF } \LDERF
  \QCPBK{f}\cplx{K}),\cplx{K}) \in \DCOH^-(S). 
  \]
  We now need to extend the quasi-isomorphism \eqref{eq:quasi_perfect}
  to $\cplx{I} \in \DQCOH^+(S)$.  First, we note that because
  $\cplx{N}$ is bounded, there
  exists an $r$ such that for all $n$ and all $\cplx{I} \in
  \DQCOH(S)$ the natural map:
  \[
  \trunc{\geq n+r}(\cplx{N}
  \tensor_{\Orb_X}^{\LDERF } \LDERF
  \QCPBK{f}\cplx{I}) \to   \trunc{\geq n+r}(\cplx{N}
  \tensor_{\Orb_X}^{\LDERF } \LDERF
  \QCPBK{f}[\trunc{\geq n}\cplx{I}]) 
  \]
  is a quasi-isomorphism. Hence, by Lemma \ref{lem:lc_der_cat} there exist maps for all
  $\cplx{I} \in \DCOH(S)$: 
  \begin{align*}
    \SRHom_{\Orb_S}(\cplx{E}_{\cplx{M},\cplx{N}},\cplx{I}) &
    \homotopic \holim{n}
    \SRHom_{\Orb_S}(\cplx{E}_{\cplx{M},\cplx{N}},\trunc{\geq-n
    }\cplx{I})\\
    &\homotopic \holim{n} \RDERF f_*\SRHom_{\Orb_X}(\cplx{M},\cplx{N}
    \tensor_{\Orb_X}^\LDERF \LDERF \QCPBK{f}[\trunc{\geq -n}\cplx{I}])\\
    &\to \holim{n} \RDERF f_*\SRHom_{\Orb_X}(\cplx{M},\trunc{\geq-
      n+r}(\cplx{N} \tensor_{\Orb_X}^\LDERF \LDERF
    \QCPBK{f}[\trunc{\geq
      -n}\cplx{I}]))\\
    &\homotopic \holim{n} \RDERF f_*\SRHom_{\Orb_X}(\cplx{M},\trunc{\geq
      -n+r}(\cplx{N}
    \tensor_{\Orb_X}^\LDERF \LDERF \QCPBK{f}\cplx{I}))\\
    &\homotopic \RDERF f_*\SRHom_{\Orb_X}(\cplx{M},\cplx{N}
    \tensor_{\Orb_X}^\LDERF \LDERF \QCPBK{f}\cplx{I}).
  \end{align*}
  Note, however, that the maps above depend on $\cplx{M}$, $\cplx{N}$,
  and $\cplx{I}$ in a non-natural way (this is because $\holim{n}$ is
  constructed as a cone, thus is not functorial). In any case,
  corresponding to the identity map $\cplx{E}_{\cplx{M},\cplx{N}} \to
  \cplx{E}_{\cplx{M},\cplx{N}}$ there is a morphism
  $\psi_{\cplx{M},\cplx{N}} : \cplx{M} \to \cplx{N}
  \tensor_{\Orb_X}^\LDERF \LDERF
  \QCPBK{f}\cplx{E}_{\cplx{M},\cplx{N}}$ (which is not necessarily
  functorial in $\cplx{M}$ or $\cplx{N}$).  Now take $\cplx{I} \in
  \DQCOH^+(S)$, then Lemma \ref{lem:big_one}\itemref{item:big_one:A}
  provides a natural sequence of maps:
  \begin{align*}
    \SRHom_{\Orb_S}(\cplx{E}_{\cplx{M},\cplx{N}},\cplx{I}) 
    &\to \RDERF \QCPSH{f} \SRHom_{\Orb_X}(\cplx{N}
    \tensor_{\Orb_X}^\LDERF \LDERF\QCPBK{f}\cplx{E}_{\cplx{M},\cplx{N}},
    \cplx{N} \tensor_{\Orb_X}^\LDERF \LDERF \QCPBK{f}\cplx{I}) \\
    &\to \RDERF \QCPSH{f} \SRHom_{\Orb_X}(\cplx{M},\cplx{N}
    \tensor_{\Orb_X}^\LDERF \LDERF\QCPBK{f}\cplx{I}).
  \end{align*}
  By \eqref{eq:quasi_perfect}, the map above is certainly a
  quasi-isomorphism for all $\cplx{I} \in \DCOH^+(S)$. To 
  show that it is a quasi-isomorphism for all
  $\cplx{I} \in \DQCOH^+(S)$, by the ``way-out right'' results of
  \cite[I.7.1]{MR0222093}, it is sufficient to prove that it is a
  quasi-isomorphism for all quasicoherent $\Orb_S$-modules. We may now
  reduce to the case where $S$ is an affine and noetherian
  scheme. Hence, it is sufficient to prove that the natural
  transformation of functors from $\QCOH{S} \to \AB$:
  \[
  \Hom_{\Orb_S}(\cplx{E}_{\cplx{M},\cplx{N}},(-)[0]) \to
  \Hom_{\Orb_X}(\cplx{M},\cplx{N}\tensor_{\Orb_X}^\LDERF \LDERF \QCPBK{f}(-)[0])
  \]
  is an isomorphism. By Lemma \ref{lem:lp_stk_hom}, both functors
  preserve filtered colimits and the exhibited natural transformation
  is an isomorphism for all $\shv{I} \in \COH{S}$. Since any $\shv{I}
  \in \QCOH{S}$ is a filtered colimit of objects of $\COH{S}$, we
  deduce the result.

  It now remains to address the compatibility of
  $\cplx{E}_{\cplx{M},\cplx{N}}$ with base change in the situation
  where $\cplx{N} \homotopic \shv{N}[0]$ and $\shv{N} \in \COH{X}$ is
  flat over $S$. So, we fix a morphism of noetherian algebraic stacks $g
  : T\to S$ such that $T$ admits a dualizing complex and form the
  $2$-cartesian square of noetherian algebraic stacks: 
  \[
  \xymatrix@-0.8pc{X_T \ar[r]^{g_X} \ar[d]_{f_T} & X \ar[d]^{f} \\ T
    \ar[r]^{g} & S.}
  \]
  By Lemma \ref{lem:base_change} and what we have proven so far, there
  is a quasi-isomorphism, natural in $\cplx{I} \in \DQCOH^+(T)$: 
  \begin{align*}
    \SRHom_{\Orb_S}(\cplx{E}_{\cplx{M},\shv{N}[0]},\RDERF g_*\cplx{I})
    \homotopic \RDERF g_*\SRHom_{\Orb_T}(\cplx{E}_{\LDERF
      \QCPBK{(g_X)}\cplx{M},(g_X^*\shv{N})[0]},\cplx{I}).
  \end{align*}
  By trivial duality \eqref{eq:sheaf_trivial_duality} we thus obtain
  a quasi-isomorphism, natural in $\cplx{I} \in \DQCOH^+(T)$: 
  \[
  \RDERF g_*\SRHom_{\Orb_T}(\LDERF \QCPBK{g}
  \cplx{E}_{\cplx{M},\shv{N}[0]},\cplx{I}) \homotopic \RDERF
  g_*\SRHom_{\Orb_T}(\cplx{E}_{\LDERF 
      \QCPBK{(g_X)}\cplx{M},(g_X^*\shv{N})[0]},\cplx{I}).
  \]
  By \eqref{eq:lghom}, we thus see that we have a quasi-isomorphism,
  natural in $\cplx{I} \in \DQCOH^+(T)$:
  \begin{equation*}
    \RHom_{\Orb_T}(\LDERF \QCPBK{g}
    \cplx{E}_{\cplx{M},\shv{N}[0]},\cplx{I}) \homotopic
    \RHom_{\Orb_T}(\cplx{E}_{\LDERF
      \QCPBK{(g_X)}\cplx{M},(g_X^*\shv{N})[0]},\cplx{I}).
  \end{equation*}  
  By Lemma \ref{lem:lc_der_cat} and the above we obtain a sequence of
  quasi-isomorphisms:
  \begin{align*}
    \LDERF \QCPBK{g} \cplx{E}_{\cplx{M},\shv{N}[0]} &\homotopic \holim{n}
    \trunc{\geq -n}\LDERF \QCPBK{g} \cplx{E}_{\cplx{M},\shv{N}[0]}
    \homotopic \holim{n} \trunc{\geq -n}\cplx{E}_{\LDERF
      \QCPBK{(g_X)}\cplx{M},(g_X^*\shv{N})[0]} \\
    &\homotopic \cplx{E}_{\LDERF
      \QCPBK{(g_X)}\cplx{M},(g_X^*\shv{N})[0]}.\qedhere
  \end{align*} 
\end{proof}
\section{Finitely generated and coherent functors}\label{sec:fg_coh}
Fix a ring $A$. Let $\FUNP{A}$ denote the category of functors 
$\MOD{A} \to \SETS$ which commute with finite products. This is a full
subcategory of the category of all functors $\MOD{A} \to \SETS$. Denote by
$\LIN{A}$ the category of $A$-linear functors $\MOD{A} 
\to\MOD{A}$. The following is straightforward. 
\begin{lem}\label{lem:runar}
  Fix a ring $A$, then the forgetful functor:
  \[
  \LIN{A} \to \FUNP{A}
  \]
 is an equivalence of categories. 
\end{lem}
In particular, for a ring $A$, set $T=\spec A$ and note that Lemma
\ref{lem:runar} shows that the category of additive functors
$\QCOH{T} \to \AB$ is equivalent to the 
category of $A$-linear functors $\MOD{A} \to \MOD{A}$. We will use
this equivalence without further mention to translate definitions
between the two categories. 

A functor $Q : \MOD{A} \to \SETS$ is \fndefn{finitely generated} if
there exists an $A$-module $I$ and an object $\eta \in Q(I)$ such
that for all $A$-modules $M$, the induced morphism of sets
$\Hom_A(I,M) \to Q(M) : f\mapsto f_*\eta$ is surjective. We 
call the pair $(I,\eta)$ a \fndefn{generator} for the functor $Q$. The
notion of finite 
generation of a functor is due to M. Auslander \cite{MR0212070}.
\begin{ex}
  Fix a ring $A$. For any $A$-module $I$, the functor $M\mapsto
  \Hom_A(I,M)$ is finitely generated. The functor $M\mapsto I\tensor_A
  M$ is finitely generated if and only if the $A$-module $I$ is
  finitely generated as an $A$-module.  A generator is obtained by
  choosing a surjection $A^n \to I$, and noting that for any
  $A$-module $M$, there is a surjection $A^n\tensor_A M \to I\tensor_A
  M$, and an isomorphism $A^n \tensor_A M \to \Hom_A(A^n,M)$. 
\end{ex}
\begin{ex}\label{ex:fg_prod}
  Fix a ring $A$ and a collection of finitely generated functors
  $\{F_\lambda : \MOD{A} \to \AB\}_{\lambda\in \Lambda}$ indexed by a
  set $\Lambda$. Then, the functor $M \mapsto \Pi_{\lambda\in
    \Lambda} F_\lambda(M)$ is finitely generated. Indeed, for each
  $\lambda\in \Lambda$, let  $(I_\lambda,\eta_\lambda)$ be a generator
  for $F_\lambda$. Then, $(\Pi_{\lambda\in \Lambda} I_\lambda,
  (\eta_\lambda)_{\lambda\in \Lambda})$ is a generator for the functor
  $M\mapsto \Pi_{\lambda \in \Lambda} F_\lambda(M)$. 
\end{ex}
For an $A$-algebra $B$, and a functor $Q : \MOD{A}
\to \SETS$, there is an induced functor $Q_B : \MOD{B} \to \SETS$
given by regarding a $B$-module as an $A$-module. Note that since the
forgetful functor $\MOD{B} \to \MOD{A}$ commutes with all limits and
colimits, it follows that if the functor $Q$ preserves certain limits
or colimits, so does 
the functor $Q_B$. In particular, from Lemma \ref{lem:runar} we see
that if the functor $Q$ commutes with finite products, then the
functor $Q_B : \MOD{B} \to \SETS$ is canonically $B$-linear and thus
defines a functor $Q_B : \MOD{B} \to \MOD{B}$. We will use this
fact frequently and without further comment. 
\begin{ex}\label{ex:fg_rest}
  For an $A$-algebra $B$, if $(I,\eta)$ generates $Q$, let $p : I \to
  I\tensor_A B$ be the natural $A$-module homomorphism, then
  $(I\tensor_A B, p_*\eta)$ generates $Q_B$.
\end{ex}
An additive functor $F:\MOD{A} \to \AB$ is
\fndefn{coherent}, if 
there exists an $A$-module homomorphism $f :I \to J$ and an element
$\eta \in F(I)$, inducing an exact sequence for any $A$-module $M$:  
\[
\xymatrix{\Hom_A(J,M) \ar[r] & \Hom_A(I,M) \ar[r] & F(M) \ar[r] & 0. }
\]
We refer to the data $(f : I \to J,\eta)$ as a \fndefn{presentation}
for $F$. For accounts of coherent functors, we refer the
interested reader to \cite{MR0212070,MR1656482}. We now have a number
of examples.  
\begin{ex}\label{ex:coherent_ab}
  Given an exact sequence of additive functors $H_i : \MOD{A} \to
  \AB$:
    \[
   \xymatrix{H_1 \ar[r] & H_2 \ar[r] & H_3 \ar[r] & H_4 \ar[r] & H_5, }
   \]
   where for $i\neq 3$, we have that $H_i$ is coherent. Then, $H_3$ is
   coherent. In particular, the category of coherent functors is stable
   under kernels, cokernels, subquotients, and extensions. This follows from
   \cite[Prop.\ 2.1]{MR0212070}.
\end{ex}
\begin{ex}\label{ex:coh_rest}
  Fix a ring $A$ and an $A$-algebra $B$. If $F : \MOD{A} \to \AB$ is a
  coherent functor, then analogously to Example \ref{ex:fg_rest}, the
  restriction $F_B : \MOD{B} \to \AB$ is also coherent. 
\end{ex}
Fix a ring $A$. A functor $F : \MOD{A} \to \MOD{A}$ is
\fndefn{half-exact} if for any short exact sequence of $A$-modules $0
\to M' \to M \to M'' \to 0$, the sequence $F(M') \to F(M) \to F(M'')$
is exact. 
\begin{ex}\label{ex:cmplx}
  Fix a ring $A$ and let $Q^\bullet$ be a complex of $A$-modules. Then
  the functor $M\mapsto H^i(\Hom_A(Q^\bullet,M))$ is coherent for all
  $i\in \Z$. If the complex $Q^\bullet$ is term-by-term projective,
  the functor $M\mapsto H^i(\Hom_A(Q^\bullet,M))$ is also
  half-exact. 
\end{ex}
An $A$-linear functor of the form $M \mapsto \Ext^i_A(Q^\bullet,M)$ is
said to be \fndefn{corepresentable by a complex}. By Example
\ref{ex:cmplx}, such functors are coherent and half-exact, and were
intially studied by M.~Auslander \cite{MR0212070}, with stronger
results---in the noetherian setting---obtained R.~Hartshorne
\cite{MR1656482}. In \cite{coh_crit_he_func}, it is shown that \'etale
locally any half-exact, coherent functor is corepresentable by a
complex.
\begin{ex}\label{ex:coh_prod}
  Fix a ring $A$ and a coherent functor $F : \MOD{A} \to \AB$. Then,
  $F$ preserves small products. It was shown by H. Krause 
  \cite[Prop.\ 3.2]{MR2026723} that the preservation of small products
  characterizes coherent functors.  
\end{ex}
\begin{ex}\label{ex:coh_lims}
  Fix a ring $A$. Example \ref{ex:fg_prod} extends to show that the
  category of coherent functors $\MOD{A} \to \AB$ is closed under
  small products. By Example \ref{ex:coherent_ab}, the category of
  coherent functors $\MOD{A} \to \AB$ is also closed under
  equalizers. Thus, the category of coherent functors $\MOD{A} \to
  \AB$ is closed under small limits. 
\end{ex}
\begin{ex}\label{ex:left_exact_coh}
  Fix a ring $A$ and a coherent functor $F : \MOD{A} \to
  \AB$ which is left exact. By Example \ref{ex:coh_prod}, $F$ also
  preserves small products, thus $F$ preserves small 
  limits. The Eilenberg-Watts Theorem \cite[Thm.\
  6]{MR0118757} now implies that $F$ is corepresentable. That is, there
  exists an $A$-module $Q$ such that $F(-) \cong \Hom_A(Q,-)$. If, in
  addition, the functor $F$ preserves direct limits, then $Q$ is
  finitely presented. By Example \ref{ex:cmplx}, this observation
  generalizes \cite[III.7.4.6]{EGA}.  
\end{ex}
\begin{ex}
  Fix a ring $A$ and an $A$-module $N$. Then the functor $M\mapsto
  M\tensor_A N$ is coherent if and only if the $A$-module $N$ is
  finitely presented.
\end{ex}
\begin{ex}\label{ex:noetheriancmplx}
  Fix a noetherian ring $R$. Let $Q^\bullet$ be a complex of finitely
  generated $R$-modules, then the functor $M\mapsto
  H^i(Q^\bullet\tensor_R M)$ is coherent and commutes with filtered
  colimits. If the complex 
  $Q^\bullet$ is, in addition, flat term-by-term, then the
  functor $M\mapsto H^i(Q^\bullet\tensor_R M)$ is also half-exact.  
\end{ex}
\begin{ex}\label{ex:noetheriantorext}
  Fix a noetherian ring $R$. Let $Q^\bullet$ be a bounded above
  complex of $R$-modules with coherent cohomology. Then, the functors
  $M\mapsto \Tor_{i}^R(Q^\bullet, M)$ and $M\mapsto 
  \Ext^i_R(Q^\bullet,M)$ are coherent, half-exact and preserve
  filtered colimits. 
\end{ex}
\section{Coherence of $\Hom$-functors: flat case}\label{sec:pf_cohstk_flat}
Proof of Theorem
  \ref{mainthms:cohstk_flat} The proof of Theorem \ref{mainthms:cohstk_flat} will be via an
induction argument that permits us to reduce to the case of Theorem
\ref{mainthms:coh_stk_dual}. The following notation will be
useful. 
\begin{notn}
  For a morphism of algebraic stacks $f : X \to S$, and $\cplx{M}$,
  $\cplx{N} \in \DQCOH(X)$, set:
  \[
  \E{\cplx{M}}{\cplx{N}} :=
  \Hom_{\Orb_X}(\cplx{M},\cplx{N}\tensor_{\Orb_X}^\LDERF
  \LDERF \QCPBK{f}(-)) : \QCOH{S} \to \AB.
  \]
  For $\cplx{N} \in \DQCOH(X)$ we set
  $\TSUB{X/S}{\cplx{N}} \subset \DQCOH(X)$ to 
  be the full subcategory having objects those $\cplx{M}$ with
  the property that $\E{\cplx{M}}{\cplx{N}[n]}$ is coherent
  for all $n\in \Z$.
\end{notn}
We begin with two general reductions. 
\begin{lem}\label{lem:cohstk_bdd_red}
  Fix an affine scheme $S$ and a morphism of algebraic stacks $f : X
  \to S$ and $\cplx{N} \in \DQCOH(X)$, then the subcategory
  $\TSUB{X/S}{\cplx{N}}\subset \DQCOH(X)$ is triangulated and closed
  under small direct sums. In particular, if
  \begin{enumerate}
  \item \label{lem:cohstk_bdd_red:item:above}
    $\DCAT^-_{\mathrm{QCoh}}(X) \subset \TSUB{X/S}{\cplx{N}}$, or
  \item \label{lem:cohstk_bdd_red:item:shv} $\cplx{N}$ has finite tor-dimension over $S$ and $\shv{M} \in \TSUB{X/S}{\cplx{N}}$ for
    all $\shv{M} \in \QCOH{X}$,
  \end{enumerate}
  then $\TSUB{X/S}{\cplx{N}} = \DQCOH(X)$.
\end{lem}
\begin{proof}
  Certainly, $\TSUB{X/S}{\cplx{N}}$ is closed under
  shifts. Next, given a triangle $\cplx{M}_1 \to \cplx{M}_2 \to
  \cplx{M}_3$ in $\DQCOH(X)$ with
  $\cplx{M}_1$, $\cplx{M}_2 \in \TSUB{X/S}{\cplx{N}}$, we
  obtain an exact sequence of functors:
  \[
  \E{\cplx{M}_2[1]}{\cplx{N}} \to
  \E{\cplx{M}_1[1]}{\cplx{N}} \to
  \E{\cplx{M}_3}{\cplx{N}} \to
  \E{\cplx{M}_2}{\cplx{N}} \to
  \E{\cplx{M}_1}{\cplx{N}}.
  \]
  By Example \ref{ex:coherent_ab},
  $\E{\cplx{M}_3}{\cplx{N}} \in
  \TSUB{X/S}{\cplx{N}}$ and so $\TSUB{X/S}{\cplx{N}}$ is
  a triangulated subcategory of $\DQCOH(X)$. Let
  $\{\cplx{M}_i\}_{i\in I}$ be a set of elements from $\TSUB{X/S}{\cplx{N}}$. Set $\cplx{M} =
  \oplus_{i\in I} 
  \cplx{M}_i$, then for all $n \in \Z$ there is an isomorphism
  of functors $\E{\cplx{M}}{\cplx{N}[n]} \cong \Pi_{i\in
    I}\E{\cplx{M}_i}{\cplx{N}[n]}$. By Example
  \ref{ex:coh_lims} we conclude that $\cplx{M} \in
  \TSUB{X/S}{\cplx{N}}$.

  For \itemref{lem:cohstk_bdd_red:item:above}, by
  \cite[Lem.~4.3.2]{MR2434692}, given $\cplx{M} \in
  \DQCOH(X)$, there is a triangle:
  \[
  \bigoplus_{n\geq 0} \trunc{\leq n}\cplx{M} \to \bigoplus_{n\geq 0}
  \trunc{\leq n} \cplx{M} \to \cplx{M}. 
  \]
  By hypothesis, $\trunc{\leq n}\cplx{M} \in
  \TSUB{X/S}{\cplx{N}}$ for all $n\geq 0$, and so
  $\cplx{M} \in \TSUB{X/S}{\cplx{N}}$.

  For \itemref{lem:cohstk_bdd_red:item:shv} by
  \itemref{lem:cohstk_bdd_red:item:above}, it is sufficient to prove that
  $\DQCOH^-(X) \subset \TSUB{X/S}{\cplx{N}}$. Since $\cplx{N}$ has
  finite tor-dimension over $S$, there exists an integer $l$ such that
  the natural map $\cplx{N} \tensor^\LDERF_{\Orb_X} \LDERF \QCPBK{f} I
  \in \DQCOH^{\geq l}(X)$ for all $I\in \QCOH{S}$. Thus, if $\cplx{M}
  \in \DQCOH^-(X)$, then for any integer $n$ we have a natural
  isomorphism of functors: $\E{\cplx{M}}{\cplx{N}[n]} \cong
  \E{\trunc{\geq l-n}\cplx{M}}{\cplx{N}[n]}$. Hence, it is sufficient
  to prove that $\DQCOH^b(X) \subset \TSUB{X/S}{\cplx{N}}$. Working
  with truncations and cones gives the result.
\end{proof}
\begin{lem}\label{lem:cohstk_maps}
  Fix an affine scheme $S$, a representable morphism of
  algebraic $S$-stacks $p : X'\to X$ which is quasicompact, and
  $\cplx{G'} \in 
  \DQCOH(X')$. Then, $\E{\cplx{M}}{\RDERF p_*\cplx{G'}} = \E{\LDERF
  \QCPBK{p}\cplx{M}}{\cplx{G'}}$. In
  particular, if $\TSUB{X'/S}{\cplx{G'}} =
  \DQCOH(X')$, then $\TSUB{X/S}{\RDERF
    p_*\cplx{G'}} = \DQCOH(X)$.
\end{lem}
\begin{proof}
  Fix $\cplx{M} \in \DQCOH(X)$, then:
  \begin{align*}
    \E{\cplx{M}}{\RDERF p_*\cplx{G'}} &=
    \Hom_{\Orb_X}(\cplx{M}, (\RDERF p_*\cplx{G'})
    \tensor_{\Orb_X}^\LDERF \LDERF \QCPBK{f}(-))\\
    &\cong \Hom_{\Orb_X}(\cplx{M}, \RDERF
    p_*(\cplx{G'} \tensor_{\Orb_{X'}}^\LDERF \LDERF \QCPBK{p}\LDERF
    \QCPBK{f}(-)))\\
    &\cong \Hom_{\Orb_X}((\LDERF \QCPBK{p}\cplx{M}),
    \cplx{G'} \tensor_{\Orb_{X'}}^\LDERF \LDERF 
    \QCPBK{g}(-))) =: \E{\LDERF \QCPBK{p}\cplx{M}}{\cplx{G'}},
  \end{align*}
  with the penultimate isomorphism given by the projection formula
  \eqref{eq:projection_formula}.  
\end{proof}

We now have our first induction result. 
\begin{lem}\label{lem:red_qcfp_F}
  Fix an affine scheme $S$, a morphism of algebraic stacks $f : X
  \to S$ which is of finite presentation, $\cplx{N} \in \DQCOH(X)$,
  and an integer $n\in \Z$. Suppose  
  that the functor $\E{\shv{M}}{\cplx{N}[r]}$ is coherent in the
  following situations:
  \begin{enumerate}
  \item for all $r<n$ and $\shv{M} \in \QCOH{X}$;
  \item $r=n$ and $\shv{M} \in \QCOH{X}$ of finite presentation.
  \end{enumerate}
  Then, the functor $\E{\shv{M}}{\cplx{N}[n]}$ is coherent for all
  $\shv{M}\in \QCOH{X}$. 
\end{lem}
\begin{proof}
  Fix $\shv{M} \in \QCOH{X}$, then we must prove that
  $\E{\shv{M}}{\cplx{N}[n]}$ is coherent. The morphism $f$ is of
  finite presentation, so by \cite[Thm.\ A]{rydh-2009}, the quasicoherent
  $\Orb_X$-module $\shv{M}$ is a filtered colimit of $\Orb_X$-modules
  $\shv{M}_\lambda$ of finite presentation. Let $\shv{Q}_1 = \oplus_\lambda
  \shv{M}_\lambda$, $\shv{Q}_2 = \oplus_{\lambda 
\leq \lambda'} \shv{M}_\lambda$ and take $\theta :
  \shv{Q}_2 \to \shv{Q}_1$ to be the natural map with $\coker \theta
  \cong \shv{M}$. Take $\cplx{Q}$ to be the cone of $\theta$ in
  $\DQCOH(X)$, for all integers $r$ we obtain an exact sequence in
  $\FUNP{S}$: 
  \[
  \xymatrix{\E{\shv{Q}_1}{\cplx{N}[r-1]} \ar[r] &
    \E{\shv{Q}_2}{\cplx{N}[r-1]} \ar[r] &
  \E{\cplx{Q}}{\cplx{N}[r]}  \ar[r] & \E{\shv{Q}_1}{\cplx{N}[r]} \ar[r] & 
  \E{\shv{Q}_2}{\cplx{N}[r]}.}
  \]
  For all integers $r$ we also have
  isomorphisms 
  \[
  \E{\shv{Q}_1}{\cplx{N}[r]} \cong \prod_\lambda \E{\shv{M}_\lambda}{\cplx{N}[r]}
  \quad \mbox{and} \quad \E{\shv{Q}_2}{\cplx{N}[r]} \cong
  \prod_{\lambda\leq \lambda'}
  \E{\shv{M}_\lambda}{\cplx{N}[r]}.
  \]
  By Examples \ref{ex:coh_lims} and \ref{ex:coherent_ab}, together
  with our hypotheses, we deduce that $\E{\cplx{Q}}{\cplx{N}[r]}$ is a
  coherent functor for all $r\leq n$. Now, there is a distinguished
  triangle $\cplx{Q} \to \shv{M}[-1] \to (\ker \theta)[1]$, thus we
  obtain an exact sequence in $\FUNP{S}$ for all integers $r$:
  \[
  \xymatrix{\E{\cplx{Q}}{\cplx{N}[r-1]} \ar[r] & \E{(\ker
      \theta)}{\cplx{N}[r-2]} \ar[r] & \E{\shv{M}}{\cplx{N}[r]} \ar[r] &
    \E{\cplx{Q}}{\cplx{N}[r]} \ar[r] & 
    \E{(\ker \theta)}{\cplx{N}[r-1]}.}
  \]
  By hypothesis, $\E{(\ker \theta)}{\cplx{N}[r]}$ is coherent for all
  $r<n$. Taking $r=n$ in the exact sequence above and applying Example
  \ref{ex:coherent_ab}, we deduce that $\E{\shv{M}}{\cplx{N}[n]}$ is
  coherent.
\end{proof}
Our next result forms the second part of the induction process. We
wish to emphasize that in the following lemma, some of the pullbacks
are underived. This is not a typographical error and is essential to
the argument.  
\begin{lem}\label{lem:red_noeth_mod}
  Fix a $2$-cartesian diagram of algebraic stacks:
  \[
  \xymatrix@-0.8pc{X \ar[r]^h \ar[d]_f & \ar[d]^{f_0} X_0\\ S \ar[r]^g & S_0}
  \]
  where $S$ and $S_0$ are affine schemes. Fix an integer $n$ and let
  $\shv{M}_0$, $\shv{N}_0 \in \QCOH{X_0}$. Assume,
  in addition, that $\shv{N}_0$ is of finite presentation, flat
  over $S_0$, and that the functors: 
  \[
  \E{\shv{M}_0}{\shv{N}_0[l]} \quad \mbox{and} \quad
  \E{\shv{F}}{(h^*\shv{N}_0)[r]}
  \]
  are coherent for all integers $l$ and for all integers $r<n$ and all
  $\shv{F} \in \QCOH{X}$. Then, the functors
  $\E{(h^*\shv{M}_0)}{(h^*\shv{N}_0)[r]}$ are coherent for
  all integers $r\leq n$. 
\end{lem}
\begin{proof}
  Set $\cplx{M} := \LDERF \QCPBK{h}\shv{M}_0$ and $\shv{N} :=
  {h}^*\shv{N}_0$. Given
  integers $r$ and $l$ set: 
  \[
  V_{r,l} := \E{\trunc{\geq -l}\cplx{M}}{\shv{N}[r+l]} \quad
  \mbox{and} \quad 
  W_{r,l} := \E{\COHO{-l}(\cplx{M})}{\shv{N}[r]}.
  \]
  It is sufficient to prove that $V_{r,l}$ is coherent for all
  $r\leq n$ and all $l$ (the result follows by taking $l=0$). Now, for
  integers $r$ and $l$ we have a distinguished triangle in 
  $\DQCOH(X)$: 
  \[
  \xymatrix{\COHO{-l-1}(\cplx{M})[l+1] \ar[r] & \trunc{\geq
      -l-1}\cplx{M} \ar[r] & \trunc{\geq -l}\cplx{M}.}
  \]
  Applying to this triangle the functor
  $\Hom_{\Orb_X}(-,\shv{N}[r+l]\tensor_{\Orb_X}^\LDERF \LDERF
  \QCPBK{f}(-))$ we obtain an exact sequence in $\FUNP{S}$:
  \begin{equation}
    \xymatrix{V_{r-2,l+1} \ar[r] & W_{r-2,l+1}\ar[r] & V_{r,l} \ar[r] &
      V_{r-1,l+1} \ar[r] & W_{r-1,l+1}.}\label{eq:relates}
  \end{equation}
  By hypothesis, $W_{r,l}$ is coherent for all $l$ and all
  $r<n$. Thus, by Example \ref{ex:coherent_ab}, the result will follow
  from the assertion that $V_{r,l}$ is coherent for all $l$ and all
  $r<n$. By induction on $r$, the exact sequence
  \eqref{eq:relates} shows that it is sufficient to prove that
  $V_{r,l}$ is coherent for all integers $l$ and all $r\leq 0$. 

  Since $\shv{N}_0$ is $S_0$-flat, the natural transformation of
  functors $\E{\trunc{\geq 0}\cplx{L}}{\shv{N}} \to
  \E{\cplx{L}}{\shv{N}}$ is an isomorphism for all $\cplx{L} \in
  \DQCOH(X)$. Hence, if $l\geq 0$ and $r\leq 0$:  
  \[
  V_{r,l} = \E{\trunc{\geq -l}\cplx{M}}{\shv{N}[r+l]} \cong 
  \E{\cplx{M}}{\shv{N}[r+l]}. 
  \]
  Moreover, if $l<0$, then $\trunc{\geq
    -l}\cplx{M} \homotopic 0$, so $V_{r,l} \equiv 0$. Thus, it remains
  to show that $\E{\cplx{M}}{\shv{N}[l]}$ is coherent for all
  $l$. Let $I\in \QCOH{S}$,
  then Lemma \ref{lem:base_change} gives a natural isomorphism
  $\E{\cplx{M}}{\shv{N}[l]}(I) \cong
  \E{\shv{M}_0}{\shv{N}_0[l]}(g_*I)$. Example \ref{ex:coh_rest} now
  gives the result.
\end{proof}
We can now prove Theorem
\ref{mainthms:cohstk_flat}.
\begin{proof}[Proof of Theorem
  \ref{mainthms:cohstk_flat}]
  We must prove that
  $\DQCOH(X)=\TSUB{X/S}{\shv{N}[0]}$. By Lemma \ref{lem:cohstk_maps},
  we may immediately reduce to the case where $f : X \to S$ is proper and of
  finite presentation. Next, by Lemma
  \ref{lem:cohstk_bdd_red}\itemref{lem:cohstk_bdd_red:item:shv}, it 
  suffices to prove that $\E{\shv{M}}{\shv{N}[n]}$ is coherent
  for all integers $n$, and all quasicoherent $\Orb_X$-modules
  $\shv{M}$.

  Now, $\E{\shv{M}}{\shv{N}[n]} \equiv 0$ for all $n<0$ and all
  $\shv{M} \in \QCOH{X}$. We now prove, by induction on $n\geq -1$,
  that $\E{\shv{M}}{\shv{N}[n]}$ is coherent. Certainly, the result
  is true for $n=-1$. Thus, we fix $n\geq 0$ and assume the result
  has been proven for all $r<n$. 

  By Lemma \ref{lem:red_qcfp_F}, it suffices to prove that
  $\E{\shv{M}}{\shv{N}[n]}$ is coherent when $\shv{M}$ is of finite
  presentation. Thus, by standard limit methods
  \cite[Prop.~B.3]{rydh-2009}, there exists an affine scheme $S_0$ of
  finite type over $\spec \Z$, a proper morphism of algebraic stacks
  $f_0 : X_0 \to S_0$, and a morphism of affine schemes $g : S \to
  S_0$ inducing an isomorphism of algebraic stacks $h : X \to X_0
  \times_{S_0} S$. This data may be chosen so that there exists
  coherent $\Orb_{X_0}$-modules $\shv{M}_0$ and $\shv{N}_0$, with
  $\shv{N}_0$ flat over $S$, as well as isomorphisms of
  $\Orb_X$-modules $h^*\shv{N}_0 \cong \shv{N}$, $h^*\shv{M}_0 \cong
  \shv{M}$. By Theorem \ref{mainthms:coh_stk_dual} and Example
  \ref{ex:noetheriantorext} the functors $\E{\shv{M}_0}{\shv{N}_0[l]}$
  are coherent for all $l$. By Lemma \ref{lem:red_noeth_mod} and the
  inductive hypothesis, the result follows. 
  \end{proof}
\section{Coherence of $\Hom$-functors: non-flat case} 
In this section we prove Theorem
  \ref{mainthms:cohstk_noeth_fd}. For this section we also retain the notation of
\S\ref{sec:pf_cohstk_flat}. We begin by dispatching Theorem
\ref{mainthms:cohstk_noeth_fd} in the projective case.
\begin{lem}\label{lem:cohstk_proj}
  Fix an affine and noetherian scheme $S$, a morphism of schemes $f :
  X \to S$ which is projective, and $\cplx{N} \in
  \DCOH^b(X)$. Then, $\TSUB{X/S}{\cplx{N}} = \DQCOH(X)$.
\end{lem}
\begin{proof}
  By Lemma \ref{lem:cohstk_bdd_red}\itemref{lem:cohstk_bdd_red:item:above}, it is sufficient to show that
  $\DQCOH^-(X) \subseteq \TSUB{X/S}{\cplx{N}}$. Since $f$ is a
  projective morphism and $S$ is an affine and noetherian scheme, $f$
  has an ample family of line bundles. It now follows from
  \cite[II.2.2.9]{MR0354655} that if $\cplx{M} \in \DQCOH^-(X)$, then
  $\cplx{M}$ is quasi-isomorphic to a complex $\cplx{Q}$ whose terms
  are direct sums of shifts of line bundles. Thus, by Lemma
  \ref{lem:cohstk_bdd_red}, it is sufficient to prove that if
  $\shv{L} \in \COH{X}$ is a line bundle, then $\shv{L}[0] \in
  \TSUB{X/S}{\cplx{N}}$. Fix $n\in \Z$, then we have
  natural isomorphisms:
  \begin{align*}
    \E{\shv{L}[0]}{\cplx{N}[n]}&:=\Hom_{\Orb_X}(\shv{L}[0],
    \cplx{N}[n]
    \tensor_{\Orb_X}^\LDERF \LDERF \QCPBK{f}(-))\\
    &\cong \COHO{0}(\RDERF \Gamma(X,(\shv{L}^{-1}[0]
    \tensor_{\Orb_X}^\LDERF \cplx{N}[n]
    \tensor_{\Orb_X}^\LDERF \LDERF \QCPBK{f}(-))\\
    &\cong \COHO{0}(\RDERF\Gamma(X,\shv{L}^{-1}[0]
    \tensor_{\Orb_X}^\LDERF \cplx{N}[n]) \tensor_{\Orb_S}^\LDERF
    (-)),
  \end{align*}
  the latter isomorphism is given by the projection formula
  \eqref{eq:projection_formula}. Since $\shv{L}$ is $\Orb_X$-flat and 
  $\cplx{N} \in \DCOH^b(X)$, it follows that
  $\shv{L}^{-1}[0]\tensor_{\Orb_X}^\LDERF \cplx{N}[n] \in  
  \DCOH^b(X)$. Hence,  $\RDERF
  \Gamma(X,\shv{L}^{-1}[0]\tensor_{\Orb_X}^\LDERF \cplx{N}[n]) \in
  \DCOH^b(S)$ \cite[III.2.2]{MR0354655}. By Example
  \ref{ex:noetheriantorext}, $\E{\shv{L}}{\cplx{N}[n]}$ is
  coherent, so $\shv{L}[0]\in \TSUB{X/S}{\cplx{N}}$.    
\end{proof}
We can now prove Theorem
\ref{mainthms:cohstk_noeth_fd}.
\begin{proof}[Proof of Theorem
  \ref{mainthms:cohstk_noeth_fd}]
    By inducting on the length of the complex $\cplx{N}$, it is
  sufficient to prove the result when $\cplx{N}$ is only
  supported in cohomological degree $0$. In particular, there exists a
  closed immersion $i : Y \to X$, such that the composition $f\circ
  i$ is proper, together with a coherent $\Orb_{Y}$-module
  $\shv{N}_0$ and a quasi-isomorphism $i_*\shv{N}_0[0] \cong
  \cplx{N}$. By Lemma \ref{lem:cohstk_maps}, it suffices to
  prove that $\TSUB{Y/S}{\shv{N}_0[0]} =
  \DQCOH(Y)$. Hence, we have reduced the claim to the
  case where the morphism $f$ is proper and where $\cplx{N}
  \homotopic \shv{N}[0]$ for some $\shv{N}\in \COH{X}$.
  
  Now let $\CSUB{X/S} \subset \COH{X}$ denote the full subcategory
  with objects those $\shv{N} \in \COH{X}$ such that
  $\TSUB{X/S}{\shv{N}[0]} = \DQCOH(X)$. By the 5-Lemma,
  it is plain to see that $\CSUB{X/S}$ is an exact subcategory (in the
  sense of \cite[III.3.1]{EGA}). We now prove by
  noetherian induction on the closed substacks of $X$ that $\CSUB{X/S}
  = \COH{X}$. By virtue of Lemma \ref{lem:cohstk_maps} and the
  technique of d\'evissage \cite[Proof of III.3.2]{EGA}, it is sufficient to
  prove that $\CSUB{X/S} = \COH{X}$ when $X$ is integral and
  $\TSUB{X/S}{\shv{Q}[0]} = \DQCOH(X)$ for all coherent
  $\Orb_X$-modules $\shv{Q}$ such that $\supp \shv{Q} \subsetneq |X|$. 

  So, we fix $\shv{N} \in \COH{X}$. Combining Chow's Lemma
  \cite[II.5.6.1]{EGA} with \cite[Thm.~2.7]{MR1844577}, 
  there exists a morphism $p : X' \to X$ that is proper, surjective, and
  generically finite such that $X'$ is a projective
  $S$-scheme. The diagonal of $X$ is finite, thus $X'':= X'\times_X X'$ is
  also a projective $S$-scheme, denote by $q
  : X'' \to X$ the induced morphism. By Lemma \ref{lem:cohstk_proj}
  we deduce that $\TSUB{X'/S}{p^*\shv{N}[0]} =
  \DQCOH(X')$ and $\TSUB{X''/S}{q^*\shv{N}[0]} = \DQCOH(X'')$. 

  Next, by Lemma \ref{lem:cohstk_maps}, $\TSUB{X/S}{\RDERF p_*
    p^*\shv{N}}=\DQCOH(X)$. Also, $p$ is generically finite, thus
  generically affine, so the support of the cohomology sheaves of
  $\trunc{\geq 1}(\RDERF p_* p^*\shv{N})$ vanishes generically. In
  particular, by noetherian induction, we deduce that
  $\TSUB{X/S}{\trunc{\geq 1}(\RDERF p_* p^*\shv{N})}=\DQCOH(X)$, thus
  $\TSUB{X/S}{p_*p^*\shv{N}[0]} = \DQCOH(X)$. An identical analysis
  for $q$ also proves that $\TSUB{X/S}{q_*q^*\shv{N}[0]} = \DQCOH(X)$
  and so $p_*p^*\shv{N}$ and $q_*q^*\shv{N} \in \CSUB{X/S}$. Hence,
  the equalizer $\tilde{\shv{N}}$ of the two maps $p_*p^*\shv{N}
  \rightrightarrows q_*q^*\shv{N}$ belongs to $\CSUB{X/S}$. Of course,
  there is also a natural map $\theta : \shv{N} \to \tilde{\shv{N}}$. But
  $X$ is integral, thus there is a dense open $U \subset X$ such that
  $p^{-1}(U) \to U$ is flat. By flat descent, $\theta$ is
  an isomorphism over $U$. So the exactness of the subcategory
  $\CSUB{X/S} \subset \COH{X}$ and d\'evissage now prove that
  $\shv{N} \in \CSUB{X/S}$.
\end{proof}
\section{Applications}
\subsection{Representability of $\Hom$-spaces}\label{sec:cohobc} 
As promised, Theorem \ref{mainthms:bdd_hom_shf} is now completely 
elementary. 
\begin{proof}[Proof of Theorem \ref{mainthms:bdd_hom_shf}]
  The latter claim follows from the former by standard limit methods.
  Now, $\Homstk_{\Orb_X/S}(\shv{M},\shv{N})$ is an \'etale sheaf,
  thus it is sufficient to prove the result in the case where $S$ is
  affine. By Theorem \ref{mainthms:cohstk_flat}, the functor:
  \[
  \Hom_{\Orb_X}(\shv{M},\shv{N}\tensor_{\Orb_X}{f}^*(-)) : \QCOH{S} \to \AB
  \]
  is coherent. The flatness of
  $\shv{N}$ over $S$ also shows that this functor is left-exact, so by
  Example \ref{ex:left_exact_coh}, it is corepresentable by a
  quasicoherent $\Orb_S$-module $Q_{\shv{M},\shv{N}}$. So, fixing an
  affine $S$-scheme $(T 
  \xrightarrow{\tau} S)$, there are natural  
  isomorphisms: 
  \begin{align*}
    \Homstk_{\Orb_X/S}(\shv{M},\shv{N})[T\xrightarrow{\tau} S] &=
    \Hom_{\Orb_{X_T}}(\tau_X^*\shv{M},\tau_X^*\shv{N}) \cong
    \Hom_{\Orb_X}(\shv{M},(\tau_X)_*\tau_X^*\shv{N})\\ 
    &\cong \Hom_{\Orb_X}(\shv{M},\shv{N}\tensor_{\Orb_X}
    {f}^*[\tau_*\Orb_T]) \\
    &\cong \Hom_{\Orb_S}(Q_{\shv{M},\shv{N}},\tau_*\Orb_T)\\
    &\cong
    \Hom_{\Orb_S\text{-}\mathrm{Alg}}(\mathrm{Sym}^\bullet_{\Orb_S}
    Q_{\shv{M},\shv{N}},\tau_*\Orb_T)\\
    &\cong \Hom_{\SCH{S}}(T\xrightarrow{\tau} S,\underline{\spec}_{\Orb_S}
    \mathrm{Sym}^\bullet_{\Orb_S} Q_{\shv{M},\shv{N}}).
  \end{align*}
  The natural isomorphism above extends to all $S$-schemes as
  $\Homstk_{\Orb_X/S}(\shv{M},\shv{N})$ is an \'etale sheaf. The
  result follows.  
\end{proof}
\subsection{Cohomology and base change}
To prove Theorem \ref{mainthm:cohobc}, we will consider some
refinements of the vanishing results of A. Ogus and G. Bergman
\cite{MR0302633} that occur in the setting of finitely generated
functors. So, we fix a ring $A$ and an $A$-linear functor $Q : \MOD{A}
\to \MOD{A}$. Define
\[
 \Van(Q) = \{ \mathfrak{p}\in \spec A \suchthat 
 Q(N) = 0 \quad \forall N\in \MOD{A_{\mathfrak{p}}} \}.    
 \]
Finitely generated functors immediately demonstrate their utility. 
\begin{prop}\label{prop:nakayama}
  Fix a ring $A$ and an $A$-linear functor $F : \MOD{A} \to
  \MOD{A}$ which preserves direct limits. If the functor $F$ is
  finitely generated, then the set $\Van(F) \subset \spec A$ is Zariski open.
\end{prop}
\begin{proof}
  Fix a generator $(I,\eta)$ for $F$ and a prime ideal $\mathfrak{p}
  \ideal A$ such that $F_{A_{\mathfrak{p}}} \equiv 0$. For $a\in A$
  there is the localisation morphism $l_a : I \to I_a$
  (resp. $l_{\mathfrak{p}} : I \to I_\mathfrak{p}$) and we set $\eta_a =
  (l_a)_*\eta$ (resp. $\eta_{\mathfrak{p}} =
  (l_\mathfrak{p})_*\eta$). Since $F(I_{\mathfrak{p}}) = 0$, it
  follows that
  $\eta_{\mathfrak{p}} = 0$. However, as $I_{\mathfrak{p}} =
  \varinjlim_{a\notin \mathfrak{p}} I_a$ and the functor $F$ commutes
  with direct limits of $A$-modules, there exists $a\notin
  \mathfrak{p}$ such that $\eta_a = 0$ in $F(I_a)$. Since the pair
  $(I_a,\eta_a)$ generates $F_{A_a}$, we have that $F_{A_a} \equiv 0$.
\end{proof}
We now record for future reference a result that is likely well-known,
though we are unaware of a reference. For an $A$-module $N$, define
$Q^N : \MOD{A} \to \SETS$ to be the 
functor $Q(-\tensor_A N) : \MOD{A} \to \SETS$.      
For another ring $C$, a $C$-module $K$, and a functor $G :
\MOD{A} \to \MOD{C}$, there is the functor $G\tensor_C K :
\MOD{A} \to \MOD{C} : M \mapsto G(M)\tensor_C K$. Given an
$A$-linear functor $H : \MOD{A} \to \MOD{A}$, and any $A$-module $N$,
there is a natural transformation of functors 
\[
\delta_{H,N} : H\tensor_A N \Rightarrow H^N.
\]
\begin{prop}\label{prop:lazardlem}
  Fix a ring $A$, and an $A$-linear functor $F : \MOD{A} \to
  \MOD{A}$. Suppose that the functor $F$ preserves direct
  limits. Then, for any flat $A$-module $M$, the natural 
  transformation
  \[
  \delta_{F,M} : F\tensor_A M \Rightarrow F^M
  \]
  is an isomorphism of functors. 
\end{prop}
\begin{proof}
  First, assume that the $A$-module $M$ is finite free. The functor
  $F$ commutes with finite products, thus the natural transformation
  $\delta_{F,M}$ induces an isomorphism. For the general case, by
  Lazard's Theorem 
  \cite{MR0168625}, we may write $M = \varinjlim_i P_i$, where each
  $P_i$ is a finite free $A$-module. Since tensor products commute
  with direct limits, as does the functor $F$, the Proposition
  follows from the case already considered.  
\end{proof}
\begin{cor}\label{cor:qf_bdd}
  Fix a noetherian ring $R$ and a bounded $R$-linear functor $G :
  \MOD{R} \to \MOD{R}$ that commutes with direct limits. 
  \begin{enumerate}
  \item For any quasi-finite $R$-algebra $R'$, the functor $G_{R'}$ is
    bounded. 
  \item For any $\mathfrak{p} \in \spec R$, the functor
    $G_{R_{\mathfrak{p}}}$ is bounded.  
  \end{enumerate}
\end{cor}
\begin{proof}
  For (1), by Zariski's Main Theorem \cite[{IV}.18.12.13]{EGA}, the
  homomorphism $R \to R'$ factors as $R \to \widetilde{R} \to R'$ where $R
  \to \widetilde{R}$ is finite and $\spec {R'} \to \spec
  \widetilde{R}$ is an open 
  immersion. Since the functor $G$ is bounded, the
  functor $G_{\widetilde{R}}$ is bounded. Thus, it suffices to consider the 
  case where $\spec R' \to \spec R$ is an open immersion. For any coherent
  $R'$-module $M'$, there exists a coherent $R$-module $M$
  together with an isomorphism of $R'$-modules $M\tensor_R R' \cong
  M'$. Since the homomorphism $R \to R'$ is
  flat, Proposition \ref{prop:lazardlem} implies that $G(M)\tensor_R
  R' \cong G(M\tensor_R R') \cong G(M')$. The functor $G$ is
  bounded, thus the $R'$-module $G(M)\tensor_R R'$ is coherent, giving
  the claim.

  For (2), fix a coherent $R_{\mathfrak{p}}$-module $N$, then for
  some $f\in R-\mathfrak{p}$ there exists an $R_f$-module $L$, together
  with an $R_{\mathfrak{p}}$-module isomorphism $L\tensor_{R_f}
  R_{\mathfrak{p}} \cong N$. By Proposition \ref{prop:lazardlem}, we have
  $G(N)\cong G(L)\tensor_{R_f} R_{\mathfrak{p}}$. By (1), since $R \to
  R_f$ is quasi-finite, it follows that $G_{R_f}$ is bounded. Thus,
  $G(N)$ is a coherent $R_{\mathfrak{p}}$-module.
\end{proof}
\begin{rem}
  Corollary \ref{cor:qf_bdd}(2) also holds for the henselization and
  strict henselization of $R_{\mathfrak{p}}$.
\end{rem}
Fix a noetherian ring $R$ and a half-exact, bounded, $R$-linear
functor $F : \COH{R} \to \COH{R}$. A. Ogus and G. Bergman show in
\cite[Thm.~2.1]{MR0302633} that if for all closed points $\mathfrak{q}
\in \spec R$ we have $F(\kappa(\mathfrak{q})) = 0$, then $F$ is the
zero functor. We have the following amplification. 
\begin{cor}\label{cor:nakayama_new} 
  Fix a noetherian ring $R$ and a bounded, half-exact, $R$-linear
  functor $G :   \MOD{R} \to \MOD{R}$ which commutes with direct
  limits. Then,  
  \[
  \Van(G) =
  \{\mathfrak{q} \in \spec R \suchthat G(\kappa(\mathfrak{q})) = 0\}.
  \]
\end{cor}
\begin{proof}
  Clearly, if $\mathfrak{q} \in\Van(G)$, then $G(\kappa(\mathfrak{q}))
  = 0$. For the other inclusion, let $\mathfrak{q} \in \spec R$
  satisfy $G(\kappa(\mathfrak{q}))=0$. By Corollary \ref{cor:qf_bdd}(2),
  the functor $G_{R_{\mathfrak{q}}}$ is bounded. Thus, \cite[Thm.\
  2.1]{MR0302633} applies, giving $G_{R_{\mathfrak{q}}} \equiv
  0$, and so $\mathfrak{q} \in \Van(G)$.
\end{proof}
An $R$-linear functor $G :
\MOD{R} \to \MOD{R}$ is \fndefn{universally bounded} if for {any}
noetherian $R$-algebra $R'$, the functor $G_{R'} : \MOD{R'} \to
\MOD{R'}$ is bounded. To combine Proposition \ref{prop:nakayama} and
Corollary \ref{cor:nakayama_new}, it is useful to have the following
easily proven Lemma at hand. 
\begin{lem}\label{lem:nakayama2}
  Fix a ring $A$ and an $A$-linear functor $F : \MOD{A} \to \MOD{A}$
  preserving direct limits. If the functor $F$ is finitely generated,
  then there exists a generator $(I,\eta)$ with $I$ a finitely
  presented $A$-module. In particular, if the ring $A$ is noetherian,
  then the functor $F$ is universally bounded.
\end{lem}
Combining Proposition \ref{prop:nakayama}, Corollary
\ref{cor:nakayama_new}, and Lemma \ref{lem:nakayama2}, we obtain the
vanishing result we desire. 
\begin{cor}\label{cor:openness_van}
  Fix a noetherian ring $R$ and an $R$-linear, half-exact functor $F :
  \MOD{R} \to \MOD{R}$ which is finitely generated and preserves
  direct limits. If $\mathfrak{q}\in \spec R$ and
  $F(\kappa(\mathfrak{q})) = 0$, then there exists $r\in
  R-\mathfrak{q}$ such that $F_{R_r} \equiv 0$. 
\end{cor}
\begin{proof}
  By Corollary \ref{cor:nakayama_new}, $\mathfrak{q} \in \Van(F)$. By
  Proposition \ref{prop:nakayama}, the set $\Van(F)$ is Zariski open,
  thus there exists $r\in R-\mathfrak{q}$ such that
  $\mathfrak{p} \in \spec R_r \subset \Van(F)$. Let $N\in \MOD{R_r}$
  and $\mathfrak{p} \in \spec R_r$, then by Proposition
  \ref{prop:lazardlem} it follows that $F(N)_{\mathfrak{p}} =
  F(N_{\mathfrak{p}})$. But $\mathfrak{p} \in \Van(F)$ and so
  $F(N)_{\mathfrak{p}} = 0$. Since $F(N)$ is an $R_r$-module, the
  result follows. 
\end{proof}
We now combine Corollary \ref{cor:openness_van} with the exchange
property proved by A. Ogus and G. Bergman
\cite[Cor.\ 5.1]{MR0302633}. Some notation: for a ring $A$, a
pair of $A$-linear functors $F_0$, $F_1 : \MOD{A} \to \MOD{A}$ is
\fndefn{cohomological} if for any exact sequence of $A$-modules $0 \to
M' \to M \to M'' \to 0$, there is a functorially induced exact
sequence of $A$-modules:
\[
\xymatrix{F_0(M') \ar[r] & F_0(M) \ar[r] & F_0(M'') \ar[r] & F_1(M')
  \ar[r] & F_1(M) \ar[r] & F_1(M'').}
\]
\begin{cor}[Property of exchange]\label{cor:prop_exch_fg}
  Fix a noetherian ring $R$ and a cohomological pair of $R$-linear 
  functors $F_0$, $F_1 :\MOD{R} \to \MOD{R}$ which  are finitely
  generated and preserve direct limits. For $i=0$, $1$ and for any
  $M\in \MOD{R}$, there is a natural map:
  \[
  \phi_i(M) :  F_i(R) \tensor_R M \to F_i(M). 
  \]
  Let $\mathfrak{q} \in \spec R$ and suppose that
  $\phi_0(\kappa(\mathfrak{q}))$ is surjective. Then,
  \begin{enumerate}
  \item there exists $r\in R-\mathfrak{q}$ such that for all $M\in
    \QCOH{R_r}$, the map $\phi_0(M)$ is an isomorphism.
  \item The following are equivalent:
    \begin{enumerate}
    \item $\phi_1(\kappa(\mathfrak{q}))$ is surjective;
    \item the $R_{\mathfrak{q}}$-module $F_1(R)_{\mathfrak{q}}$ is
      free. 
    \end{enumerate}
  \end{enumerate}
\end{cor}
\begin{proof}
  By Lemma \ref{lem:nakayama2}, the functors $F_i$ are bounded. Now
  apply \cite[Cor.\ 5.1]{MR0302633} and Corollary
  \ref{cor:openness_van}. 
\end{proof}
\begin{proof}[Proof of Theorem \ref{mainthm:cohobc}]
  Throughout we fix a $2$-cartesian diagram of noetherian algebraic
  stacks:
  \[
  \xymatrix@-0.8pc{X_T \ar[r]^{g_X} \ar[d]_{f_T} & X \ar[d]^f \\ T
    \ar[r]^g & S.}
  \]
  The results are all smooth local on $S$ and $T$, thus we may assume
  that $S=\spec A$ and $T=\spec B$ and $g$ is induced by a ring
  homomorphism $A \to B$. We may also work with the global
  $\Ext$-groups instead of the relative $\SExt$-sheaves. 

  For an $A$-module $I$ set $E^q(I) =
  \Ext^q_{\Orb_X}(\cplx{M},\shv{N}\tensor_{\Orb_X} f^*I)$, which gives
  an $A$-linear functor $\MOD{A} \to \MOD{A}$. By Theorem
  \ref{mainthms:cohstk_flat}, the functor $E^q$ is coherent and by
  Lemma \ref{lem:lp_stk_hom} the functor preserves filtered colimits
  (in particular, it preserves direct limits). By Lemma
  \ref{lem:base_change}, if $J$ is a $B$-module, there is a natural
  isomorphism:
  \[
  E^q(J) = \Ext^q_{\Orb_X}(\cplx{M},\shv{N} \tensor_{\Orb_X} f^*g_*J)
  \cong \Ext^q_{\Orb_{X_T}}(\LDERF \QCPBK{(g_X)}\cplx{M}, g_X^*\shv{N}
  \tensor_{\Orb_{X_T}} f_T^*J).
  \]
  In particular, taking $J=B$ we obtain a natural map:
  \[
  E^q(A)\tensor_A B := \Ext^q_{\Orb_X}(\cplx{M},\shv{N})\tensor_A B
  \to \Ext^q_{\Orb_{X_T}}(\LDERF \QCPBK{(g_X)}\cplx{M}, g_X^*\shv{N})
  = E^q(B).   
  \]
  The result now follows from Corollary \ref{cor:prop_exch_fg}. 
\end{proof}
\bibliography{bibtex_db/references}
\bibliographystyle{bibtex_db/dary}
\end{document}